\documentclass{amsart}
\usepackage{amssymb}
\usepackage{graphics}

\theoremstyle{plain}

\newtheorem*{theorem}{Theorem}

\def    \R  {{\Bbb R}}
\def    \Z  {{\Bbb Z}}

\begin{document}
\title[The fundamental groups of contact toric manifolds]{The fundamental groups of contact toric manifolds}

\author{Hui Li} 
\address{School of mathematical Sciences\\
        Soochow University\\
        Suzhou, 215006, China.}
        \email{hui.li@suda.edu.cn}

\thanks{2010 classification. Primary:  53D10, 53D20; Secondary:
55Q05}
\keywords{contact toric manifold, Reeb type, contact moment map, moment cone, symplectic toric manifold.}

\begin{abstract}
Let $M$ be a connected compact contact toric manifold.  Most of such manifolds are of Reeb type.  We show that if  $M$ is of Reeb type, then 
$\pi_1(M)$ is finite cyclic, and we describe how to obtain the order of  
$\pi_1(M)$ from the moment map image.
\end{abstract}

 \maketitle
Let $M$ be a contact manifold, and $\alpha$ be a contact $1$-form on 
$M$. Let $T^k$ be a connected compact 
$k$-dimensional torus. If $T^k$ acts on $M$  preserving the contact form $\alpha$, then it preserves
the contact structure $\xi=ker (\alpha)$. 
 
A contact manifold $M$ of dimension $2n+1$ with an effective $T^{n+1}$-action preserving the contact structure is called a {\bf contact toric manifold}. If the Reeb vector field of a contact form on $M$ is generated by a one parameter subgroup action of $T^{n+1}$, then the contact $T^{n+1}$-manifold $M$ is called a {\bf contact toric manifold of Reeb type}.

Recall that a $2n$-dimensional symplectic manifold equipped with an effective Hamiltonian $T^n$-action  is called a symplectic toric manifold. Contact toric manifolds are the odd dimensional analog of symplectic toric manifolds. Compact symplectic toric manifolds and compact contact toric manifolds are both classified (\cite{D} and \cite{Lc}). Most of the compact contact toric manifolds are of Reeb type.

Compact symplectic toric manifolds are simply connected 
(\cite{A} p235, \cite{L0, L}).  In contrast, 
the fundamental groups of
connected compact contact toric manifolds are finite abelian if they are of Reeb type (\cite{Lh}), and are infinite abelian if they are not of Reeb type.
(The latter fact can be derived by listing the non-Reeb type manifolds using
the classification in \cite{Lc}.)

In this paper, we prove a result on the fundamental groups of compact contact toric manifolds of Reeb type. To describe the result, we define some
terms and state a known result as follows. 
Let $(M, \alpha)$ be a connected compact contact toric manifold of dimension $2n+1$. Let $\mathfrak t$ be the Lie algebra of the torus $T^{n+1}$, and $\mathfrak t^*$ be the dual Lie algebra. The {\bf contact moment map}  $\Phi\colon M\to\mathfrak t^*$ is defined to be
$$\langle\Phi(x), X\rangle = \alpha_x\big(X_M(x)\big), \,\, \forall\, x\in M,\, \mbox{and}\,\,\, \forall \,\, X\in\mathfrak t,$$
 where $X_M$ is the vector field on $M$ generated by the $X$-action.
The {\bf moment cone} of $\Phi$ is defined as
$$C(\Phi) =\big\{t\Phi(x)\,|\, t\geq 0,\, x\in M\big\}.$$
It is known (\cite{{BG}, {Lh}, {Lc}} etc.) that, if $M$ is of Reeb type, 
then $C(\Phi)$ is a {\it strictly convex rational good polyhedral cone} (of dimension $n+1$). {\it Strictly convex} means that $C(\Phi)$ contains no linear subspaces of $\mathfrak t^*$ of positive dimension, {\it polyhedral} means that $C(\Phi)$ is a cone over a polytope, {\it rational} means that the normal vectors of the facets of the cone  lie in the integral lattice of $\mathfrak t$, and 
 {\it good} means that for any codimension $l$ face $\mathcal F_l$ of $C(\Phi)$, the normal vectors of the facets which intersect at $\mathcal F_l$ form a $\Z$-basis of the lattice of an $l$-dimensional linear subspace of $\mathfrak t$.

\begin{theorem}
Let $(M,  \alpha)$ be a connected compact contact toric manifold of Reeb type with dimension $2n+1$. Let 
$$I=\big\{v_1, v_2, \cdots, v_d\big\}$$ be the
set of primitive inward normal vectors of the facets of the moment cone, ordered in the way that the first $n$ vectors  are the normal vectors of the facets which intersect at a ({\it any}) $1$-dimensional face of the moment cone.  Then
$$\pi_1(M)=\Z_k, \,\,\,\mbox{where}$$
$$ k=\gcd\big(\det [v_1, v_2, \cdots, v_n, v_{n+1}], \det [v_1, v_2, \cdots, v_n, v_{n+2}], \cdots,  \det [v_1, v_2, \cdots, v_n, v_d]\big).$$
\end{theorem}

The $3$-dimensional lens spaces are compact contact toric manifolds of Reeb type. Hence any finite cyclic group can be the fundamental group of a contact toric manifold of Reeb type.

\begin{proof}[Proof of Theorem]
Let $\Z_{T}\subset \mathfrak t$ be the integral lattice of the torus $T^{n+1}$, and $\mathcal L$ the sublattice of  $\Z_{T}$ generated by the elements in $I$.  By Lerman's Theorem \cite{Lh},  
$$\pi_1(M)= \Z_{T}/\mathcal L.$$

We identify $\Z_{T}=\Z^{n+1}\subset\R^{n+1}$. Since the moment cone $C(\Phi)$ is a good cone,
$v_1, \cdots, v_n$ is a $\Z$-basis of an $n$-dimensional subspace
of $\Z^{n+1}$. So there exists another  vector $u\in \Z^{n+1}$ such that $\{v_1, \cdots, v_n, u\}$ forms a $\Z$-basis of  $\Z^{n+1}$. Let  $\mathcal L'$ be the sublattice generated by the elements in $\{v_1, \cdots, v_n\}$. Then
$$\Z_{T}/\mathcal L' =\Z^{n+1}/\Z^n = \Z = \Z \langle u\rangle.$$
Since $\{v_1, \cdots, v_n, u\}$ is a $\Z$-basis of  $\Z^{n+1}$, for $\forall \,\, n+1\leq j\leq d$, we have
 $$v_j = l_j u \mod \mathcal L', \,\,\,\mbox{where $l_j\in\Z$}.$$
Let  $k = \gcd (l_j)_{j=n+1}^d$. Since the elements in $I$ span an $n+1$-dimensional vector space, 
at least one $l_j \neq 0$. So $k\neq 0$.
Then
$$\Z_{T}/\mathcal L = \Z \langle u\rangle/ k\Z \langle u \rangle = \Z_k$$
is finite cyclic. Moreover, notice that 
$$l_j =\pm \det [v_1, \cdots, v_n, v_j], \,\,\,\forall \,\,\, n+1 \leq j\leq d,$$
where  $[v_1, \cdots, v_n, v_j]$ denotes the matrix with column vectors $v_1, \cdots, v_n$ and $v_j$.
\end{proof}

\subsubsection*{Acknowledgement}  
I thank Reyer Sjamaar for a remark on the moment cone of a compact connected contact toric manifold of Reeb type,  
which helped me to shorten a proof I tried earlier.

This work is supported by the NSFC grant K110712116.

\end{document}